\documentclass[12pt,a4paper,twoside]{amsart}
\usepackage{amsfonts, amsthm, amsmath, amssymb}
\usepackage{hyperref}
\usepackage{mathrsfs,amsmath}
\hypersetup{colorlinks=false}

\usepackage[margin=1.25in]{geometry}

\usepackage{helvet}
\usepackage{hyperref}

\usepackage{amsthm}
\usepackage{xcolor}
\newtheorem{theorem}{Theorem}
\newtheorem{lemma}{Lemma}[section]
\newtheorem{remark}{Remark}

\newtheorem{proposition}{Proposition}
\newtheorem{corollary}{Corollary}

\begin{document}
	
	\author{ Sumit Kumar, Kummari Mallesham, Prahlad Sharma and Saurabh K Singh}
	\title{Moments of  derivatives of  Modular  $L$-functions }

	\address{  Sumit Kumar, Prahlad Sharma \newline {\em Alfr\'ed R\'enyi Institute of Mathematics, Budapest, Hungary; \newline  Email:  sumit@renyi.hu; prahlad@renyi.hu
	} }
	\address{ Kummari Mallesham \newline {\em Indian Institute of Technology, Bombay,  India; \newline  Email: iitm.mallesham@gmail.com
	} }
	\address{ Saurabh Kumar Singh \newline {\em Indian Institute of Technology, Kanpur, India; \newline  Email: skumar.bhu12@gmail.com
	} }
	\maketitle
	
	\begin{abstract}
		Let $f$ be a Hecke eigenform  for  the group $\Gamma_{0}(q)$ and $\chi_{d}$ be a primitive quadratic character of conductor $|d|$.  In this article, we prove an asymptotic for the second moment of  the derivative of 
		$L(s, f \otimes \chi_{8d})$ at the central point $1/2$, which was previously known under   GRH by Petrow  \cite{petrow}. 
		%
	\end{abstract}


	\section{Introduction}  
	Let $f$ be   a Hecke eigenform for  the group $\Gamma_{0}(q)$  of even weight $k$ with trivial central character.  We write its Fourier expansion as
	%
	%
	$$f(z) = \sum_{n=1}^{\infty} \lambda_{f}(n) \, n^{(k-1)/2} \, e(nz), \, \, z\in \mathbb{H}.$$
	with  $\lambda_{f}(1)=1$ and $f$ has been normalised so that  the Deligne's bound gives  $ \vert \lambda_{f}(n) \vert \leq \tau(n)$ for all $n \geq 1$,  where $\tau(n)$ is the divisor function.
	Let $d$ be a fundamental discriminant coprime to $q$, and let $\chi_{d}(\cdot)= \left(\frac{d}{\cdot}\right)$ denote the primitive quadratic character of conductor $|d|$. Then $f \otimes \chi_{d}$ is a newform on $\Gamma_{0}(q|d|^{2})$ (though $f$ need not be a newform) and  the twisted $L$-function is given by
	$$L(s, f \otimes \chi_{d}) = \sum_{n=1}^{\infty} \frac{\lambda_{f}(n) \, \chi_{d}(n)}{n^{s}}, \, \, \Re(s) >1.$$
	It 	has an  analytic continuation to the whole of $\mathbb{C}$ and satisfies the following  functional equation
	$$\Lambda(s, f \otimes \chi_{d}) = i^{k} \eta \chi_{d}(-q) \, \Lambda(1-s, f\otimes \chi_{d}),$$
	where
	$$\Lambda(s, f\otimes \chi_{d}) = \left(\frac{|d|\sqrt{q}}{2 \pi}\right)^{s} \Gamma\left(s+\frac{k-1}{2}\right) \, L(s, f\otimes \chi_{d}),$$
	and $\eta$ is   the eigenvalue of the Fricke involution which is always $\pm 1$. We denote the root number by $\omega(f\otimes \chi_{d}) := i^{k} \eta \chi_{d}(-q)$. Note that if $d$ is a fundamental discriminant, then $\chi_{d}(-1)= 1$ if $d$ is positive and $\chi_{d}(-1)=-1$ if $d$ is negative.  The derivative of the $L$-function $L(s, f \times \chi_d)$ is given by 
	$$L^{\prime}(s, f \otimes \chi_{d} ) = -\sum_{n=1}^{\infty} \frac{\lambda_{f}(n) \chi_{d}(n) \log n}{n^{s}}, \, \, \Re(s) >1,$$
	which  also has a 
	functional equation 
	$$\Lambda^{\prime}(s, f \otimes\chi_{d}) = - \omega(f\otimes \chi_{d}) \, \Lambda^{\prime} (1-s, f\otimes \chi_{d}).$$
	In this paper, we are interested in the asymptotic behaviour  of  $L^\prime(s, f \otimes \chi_{d})$ at the central point $1/2$ over  the family of quadratic twists.  The mean value of this  family has been explored  previously by many authors,   particularly by Bump-Friedberg-Hoffstein \cite{BFH}, Murty-Murty \cite{murty}, Iwaniec \cite{I}, Munshi  \cite{M1}, \cite{M2} and Petrow \cite{petrow}.  Petrow, assuming GRH, proved
	\begin{equation} \label{L^2 asy}
		\sideset{}{^\star}{\sum}_{\substack{ (d,2q)=1 \\ \omega(f \otimes \chi_{8d})=-1}} \vert L^{\prime}\left(1/2, f\times \chi_{8d}\right) \vert^{2} J\left(\frac{8d}{X}\right) \sim  C_f^\prime  X\log^3X,
	\end{equation}
	where $C_f$ is a fixed explicit contant depending on $f$ only and  $J: \mathbb{R} \to \mathbb{R}_{\geq 0}$ is a compactly supported smooth function having support  in $[1/2,2]$.  This result was motivated by an important work by Soundarajan-Young \cite{SY}, who proved an analog of \eqref{L^2 asy} without the derivative, i.e., 
	\begin{equation} 
		\sideset{}{^\star}{\sum}_{\substack{ (d,2q)=1 }} \vert L\left(1/2, f\times \chi_{8d}\right) \vert^{2} J\left(\frac{8d}{X}\right) \sim  C_f  X\log X,
	\end{equation}
	assuming GRH. The above result was proved unconditionally  recently in a breakthrough work by Li \cite{Li}. The main aim of this article is prove \eqref{L^2 asy} unconditionally.  
	\begin{theorem} \label{mainth}
		Let $f$ and $J$ be as above. Then we have
		$$ \sideset{}{^\star}{\sum}_{\substack{ (d,2q)=1 \\ \omega(f \otimes \chi_{8d})=-1}} \vert L^{\prime}\left(1/2, f\times \chi_{8d}\right) \vert^{2} J\left(\frac{8d}{X}\right) \, = \,  C_{f} \, \widetilde{J}(1) \, X \log ^{3}X +O \left(X\, (\log X)^{\frac{5}{2}+\epsilon}\right),$$
		where $C_{f}$ is some explicit constant depending  only on $f$ and $\widetilde{J}$ is the Mellin transform of $J$. 
	\end{theorem}
	
	We  combine the methods of  Li \cite{Li} and  Petrow \cite{petrow} to prove  the above  theorem. A similar result holds when $\omega(f \otimes \chi_{8d})=1$. 
	
	\begin{corollary} \label{totaldis}
		We have
		$$ \sideset{}{^\star}{\sum}_{\substack{ (d,2q)=1 }} \vert L^{\prime}\left(1/2, f\times \chi_{8d}\right) \vert^{2} J\left(\frac{8d}{X}\right) \, = \,  \mathcal{C}_{f} \, \widetilde{J}(1) \, X \log ^{3}X +O\left( X\, (\log X)^{\frac{5}{2}+\epsilon}\right),$$
		where $\mathcal{C}_{f} $  is some constant depending  on $f$ only. 
	\end{corollary}
	
	\begin{proof}
		When  $\omega(f\otimes\chi_{8d})=1$, it follows from the functional equation 
		$$\Lambda^{\prime}(s, f \otimes \chi_{8d}) = -\, \Lambda^{\prime}(1-s, f\otimes \chi_{8d}),$$
		that $\Lambda^{\prime}(1/2,f\otimes \chi_{8d}) =0$. Thus
		$$L^{\prime}({1}/{2}, f\otimes \chi_{8d}) = -L({1}/{2},f\otimes\chi_{8d}) \Big\{ \log\left(\frac{8|d| \sqrt{q}}{2\pi}\right) \,  +\frac{\Gamma^{\prime}\left({k}/{2}\right) }{\Gamma\left({k}/{2}\right) }\Big \}.$$
		On applying the asymptoic expression 
		$$\sum_{(d,2q)=1} \left( L\left({1}/{2}, f\otimes \chi_{8d}\right) \right)^{2} \, J\left(\frac{8d}{X}\right) = C^{\prime}_{f} \, \widetilde{J}(1) \, X \log X +O\left(X (\log X)^{1/2+\epsilon}\right)$$
		from  Li \cite[Theorem 1.1]{Li}, it follows that
		\begin{align*}
			&\sideset{}{^\star}{\sum}_{\substack{ (d,2q)=1 \\ \omega(f \otimes \chi_{8d})=1}} \vert L^{\prime}\left(1/2, f\times \chi_{8d}\right) \vert^{2} J\left(\frac{8d}{X}\right) \, \\
			&	= (\log X)^{2}\sum_{ (d,2q)=1}  \left( L\left({1}/{2}, f\otimes \chi_{8d}\right) \right)^{2} \, J\left(\frac{8d}{X}\right)+O\left(X \log X\right) \\
			&= C_{f}^{\prime}\, \widetilde{J}(1) \, X (\log X)^{3} +O\left(X (\log X)^{\frac{5}{2}+\epsilon}\right).
		\end{align*}
		Thus  the corollary follows by  Theorem \ref{mainth}  and the above expression. 
	\end{proof}
	A similar method works for all discriminants $d$ as well.  As an application, we improve upon  the lower bound for 
	\begin{align*}
		N(X) = \#  \{ 0< d\leq X:   d  \  \mathrm{is \ a \ fundamental \ discriminant   \ with } \   (d,q)=1 \\  \text{and} \  L^{\prime}(1/2,E \otimes\chi_{d} \neq 0) \}.
	\end{align*}
	Here $E$ be a modular elliptic curve over $\mathbb{Q}$ with  the $L$-function $L(s, E )=L(s, f ) $ for some eigenform $f$ of weight $2$ and level $q$, where $q$ is the conductor  of $E$. Iwaniec  \cite{I} proved that $N(X) \gg_\epsilon X^{2/3-\epsilon}$ for any $\epsilon>0$.   This bound was later improved to  
	\begin{align} \label{NX}
		N(X) \gg_\epsilon X^{1-\epsilon},
	\end{align}
	by Perelli and Pomykala \cite{parelly}. We improve upon  \eqref{NX} in the following corollary. 
	\begin{corollary}\label{th1}
		We have 
		$$N(X) \gg X/\log X. $$
	\end{corollary}
	\begin{proof}
		Main ingredients to prove Corollary \ref{th1} are  non-trivial  lower bound for the  first moment of  $L^\prime(1/2, f\otimes \chi_d)$ and an   upper bound for the second moment of  $L^\prime(1/2, f\otimes \chi_d)$ over $d$ (see \cite{parelly}). Indeed,  by  the work of Murty-Murty \cite{murty} and Iwaniec \cite{I}, we have 
		\begin{equation}\label{L^}
			\sideset{}{^\star}{\sum}_{(d,2q)=1}L^{\prime}({1}/{2},f\otimes \chi_{d}) \sim C X \, \log X,
		\end{equation}
		for some constant $C \neq 0$, and by Corollary  \ref{totaldis} we have 
		\begin{equation} \label{L^2}
			M^\prime(2):=\sideset{}{^\star}{\sum}_{(d,2q)=1}|L^{\prime}({1}/{2},f\otimes \chi_{d})|^2 \ll_{f}   X\log^3X. 
		\end{equation}
		Thus on    applying  Cauchy-Schwartz inequality   to the left-hand side of \eqref{L^}, and then using \eqref{L^2}, we have the corollary. 
	\end{proof}
	\section{Preliminaries}

	\begin{lemma}
		We have 
		$$L^{\prime}(1/2, f \otimes \chi_{d}) \, \, \mathbb{I}(\omega(f \otimes \chi_{d})=-1) = \left(1-i^{k} \eta \chi_{d}(-q)\right)\sum_{n=1}^{\infty} \frac{\lambda_{f}(n) \chi_{d}(n)}{n^{1/2}} \, W\left( \frac{n}{|d|}\right),$$
		where
		$$W(y)= \frac{1}{2 \pi i} \int_{(3)} \frac{\Gamma\left(u+{k}/{2}\right)}{\Gamma\left({k}/{2}\right)} \left(\frac{2 \pi y}{\sqrt{q}}\right)^{-u} \, \frac{du}{u^{2}}.$$
	\end{lemma}
	
	\begin{proof}
		See Lemma 3.1 of \cite{petrow}.
	\end{proof}
	
	\begin{lemma} \label{G}
		There exists a smooth function $G: \mathbb{R} \to \mathbb{R}_{\geq 0}$ such that
		\begin{itemize}
			\item Support of $G$ lies in $[3/4,2]$
			\item $G(x) =1 $ for all $x \in [1,3/2]$
			\item $G(x)+G(x/2)=1$ for all $x \in [1,3]$.
		\end{itemize}
	\end{lemma}
	
	\begin{proof}
See   \cite[(2.11)]{Li}. 
	\end{proof}
	
	Let $H$ be a positive integer. Then the function $F(x)$ defined by
	$$F(x) = G(x)+G(x/2)+\ldots+G(x/2^{H}),$$
	satisfies $F(x)=1$ for all $x \in [1, 3. 2^{H-1}]$ and is  supported in $[3/4, 2^{H+1}]$. 
	
	\begin{remark} \label{V}
		\begin{itemize}
			\item In particular the function
			$$V(x) = G(2x) + G(x) + G(x/2)$$ satisfies $V(x)=1$ for all $x \in [1/2,3]$.
			\item We have the smooth partition of unity
			$$\sum_{H=0}^{\infty} G\left(\frac{x}{2^{H}}\right)=1$$
			for all $x \in [1, \infty)$.
		\end{itemize}
	\end{remark}

	The following proposition is due to Xiannan Li \cite[Proposition 3.2]{Li}.
	\begin{proposition} \label{mainpropo}
		Let the smooth function $G(x)$ as above. Then we have 
		$$\sideset{}{^\star}{\sum}_{M \leq |m| \leq 2M} \Big \vert \sum_{n=1}^{\infty} \frac{\lambda_{f}(n)}{n^{1/2+it}} \, \left(\frac{m}{n}\right) G\left(\frac{n}{N} \right)\Big \vert^{2} \ll_{f} \left(1+|t|\right)^{2} \left(M + N \, \log \left(2+\frac{N}{M}\right)\right).$$
	\end{proposition}
	
	We record the following lemma from Xiannan Li \cite[Lemma 6.3]{Li}.
	\begin{lemma} \label{coprimelargesieve}
		For any $\mathcal{Y}, N \geq 1$, real $t$, positive integer $\ell$, we have
		$$\sum_{\substack{(d,2)=1 \\ d \leq \mathcal{Y}}} \Big \vert \sum_{ (n,\ell)=1} \frac{\lambda_{f}(n)}{n^{\frac{1}{2}+it}} \, \left(\frac{8d}{n}\right)  \, G\left(\frac{n}{N}\right)\Big \vert^{2} \ll \tau(d)^{5} \, \mathcal{Y} \left(1+|t|\right)^{3} \log \left(2+|t|\right),$$
		where $\tau$ is the divisor function.
	\end{lemma}
	
	We record here the following lemma from Xiannan Li \cite[Lemma 2.4]{Li}
	\begin{lemma} \label{Poisson}
		Let $H : \mathbb{R}_{+} \to \mathbb{R} $ be a Schwartz class function.  Let $n$ be an odd integer. Then we have
		\begin{align*}
			\sum_{ (d,2)=1} \left(\frac{8d}{n}\right) \, H\left(\frac{d}{X}\right) &= \delta_{\square}(n) \, \frac{X}{2} \, \check{H} (0) \, \prod_{p \mid n} \left(1-\frac{1}{p}\right) \\
			&+ \frac{X}{2} \sum_{ k \neq 0} (-1)^{k} \frac{G_{k}(n)}{n} \, \check{H}\left(\frac{X k}{2n}\right),
		\end{align*}
		where $\check{H}$ is Fourier-type transform which defined as
		$$\check{H}(y) = \int_{-\infty}^{\infty} \left(\cos\left(2 \pi xy + \sin\left(2 \pi xy\right)\right)\right) H(x) \, dx.$$
		Moreover, for $y \neq 0$ we also have
		$$\check{H}(y)= \frac{1}{ 2 \pi i} \int_{(1/2)} \tilde{H} (1-s) \, \Gamma(s) \left(\cos +\frac{y}{|y|} \sin\right)\left(\frac{\pi s}{2}\right) \, \left(2 \pi |y|\right)^{-s} \, ds,$$
		where
		$$\tilde{H}(s) = \int_{0}^{\infty} H(x) x^{s-1} \, dx$$
		is the Mellin transform of $H$.
		And $$G_{k}(n) = \left(\frac{1-i}{2}+\left(\frac{-1}{n}\right) \frac{1+i}{2}\right) \sum_{a \, \rm mod \, n} \left(\frac{a}{n}\right) \, e\left(\frac{ak}{n}\right)$$
		is the Gauss like sum. 
	\end{lemma}
	
	In the following lemma we record properties of $G_{k}(n)$.
	\begin{lemma} \label{Gausslemma}
		For $m,n$ relatively prime odd integers, we have
		$$G_{k}(mn) =G_{k}(m) G_{k}(n)$$
		and for $p^{\alpha} || k $ (set $\alpha = \infty$ for $k=0$), then
		\[
		G_{k}(p^{\beta}) =
		\begin{cases}
			0, \quad  &\text{if} \quad \beta \leq \alpha \, \, \, \text{is odd} \\
			\phi(p^{\beta}), \quad  &\text{if} \quad \beta \leq \alpha \, \, \, \text{is even}\\
			-p^{\alpha}, \quad  &\text{if} \quad \beta = \alpha+1 \, \, \, \text{is even}\\
			\left(\frac{k p^{-\alpha}}{p}\right) p^{\alpha} \, \sqrt{p}, \quad  &\text{if} \quad \beta = \alpha +1\, \, \, \text{is odd} \\
			0, \quad  &\text{if} \quad \beta \geq  \alpha+2 \, \, \, \text{is odd}.
		\end{cases}
		\] 
	\end{lemma}
	
	\section{Proof  of  Theorem \ref{mainth}}
	We are interested in  getting an  asymptotic formula for the smoothed sum
	$$M^\prime_{\text{s}} (2)=\sideset{}{^\star}{\sum}_{\substack{(d,2q)=1 \\ \omega(f \otimes \chi_{8d})=-1}} \vert L^{\prime}\left(1/2, f\otimes \chi_{8d}\right) \vert^{2} \, J\left(\frac{8 d}{X} \right),$$
	where $J$ is a non-negative smooth function suppored in $[1/2, 2]$. Following Lemma 3.1 of Petrow \cite{petrow}, we express $L^\prime(1/2, f \otimes \chi_{8d} )$ as a Dirichlet series using the approximate functional equation 
	$$L^{\prime}(1/2, f \otimes \chi_{8d}) \, = \left(1-i^{k} \eta \chi_{8d}(-q)\right)\sum_{n=1}^{\infty} \frac{\lambda_{f}(n) \chi_{d}(n)}{n^{1/2}} \, W\left( \frac{n}{|8d|}\right),$$
	which holds if $\omega(f \otimes \chi_{8d})=-1$. Here the cut-off function $W$ is given by 
	
	$$W(y)= \frac{1}{2 \pi i} \int_{(3)} \frac{\Gamma\left(u+{k}/{2}\right)}{\Gamma\left({k}/{2}\right)} \left(\frac{2 \pi y}{\sqrt{q}}\right)^{-u} \, \frac{du}{u^{2}}.$$
	Next we split the Dirichlet series  into the main part  and the  tail part.  To this end, we set
	$$\mathcal{A}(8d):= \mathcal{A}(1/2, f \otimes \chi_{8d}) = (1-i^{k} \eta \chi_{8d}(q))\sum_{n=1}^{\infty} \frac{\lambda_{f}(n) \chi_{8d}(n)}{n^{1/2}} \, W\left( \frac{n}{\mathcal{M}}\right),$$
	where $\mathcal{M} = X/(\log X)^{1000}$, and 
	$$\mathcal{B}(8d):= \mathcal{B}(1/2, f\otimes \chi_{8d}) = L^\prime(1/2, f \otimes \chi_{8d}) - \mathcal{A}(1/2, f \otimes \chi_{8d}).$$
	Theorem \ref{mainth} follows from the following propositions.
	\begin{proposition} \label{errorbound}
		We have
		\begin{equation}\notag
			\sideset{}{^\star}{\sum}_{(d, 2q)=1} \Big \vert \mathcal{B}(1/2, f\otimes \chi_{8d}) \Big \vert^{2} \, J\left(\frac{8d}{X}\right)\ll_{f,\epsilon}  X(\log X)^{2} (\log \log X)^4.
		\end{equation}
	\end{proposition}
	\begin{proposition} \label{aasymptotic}
		We have
		$$\sideset{}{^\star}{\sum}_{(d,2q)=1}  \Big \vert \mathcal{A}(1/2, f \otimes \chi_{8d})\Big \vert^{2} \, J\left(\frac{8d}{X}\right)=C_{f} \, \check{J}(0) \, X  \, (\log X)^{3} + O\left(X\, (\log X)^{2}\right),$$
		where $C_{f}$ is a constant  depending  on $f$.
	\end{proposition}
	Indeed,  on applying the identity $(a+b)^2=a^2+b^2+2ab$  to $M^\prime_{\text{s}}(2)$ and then   Cauchy-Schwartz inequality followed by the above propositions, we get Theorem \ref{mainth}. 
	In the rest of the paper, we prove Proposition \ref{errorbound} (in  Section \ref{propproofsec1}) and  Proposition \ref{aasymptotic} (in Section \ref{asmptoicsection}).

	\section{Proof of Proposition \ref{errorbound}} \label{propproofsec1}

	Recall that  $\mathcal{B}(1/2, f\otimes \chi_{8d})$ is given by
	$$\left(1-i^{k} \eta \chi_{d}(-q)\right) \, \frac{1}{2 \pi i} \int_{(3)} \frac{(2 \pi/\sqrt{q})^{-w} \, \Gamma\left(\frac{k}{2}+w \right)}{\Gamma\left(\frac{k}{2}\right)} \frac{(8d)^{w}-\mathcal{M}^{w}}{w^{2}} \sum_{n=1}^{\infty} \frac{\lambda_{f}(n) \chi_{8d}(n)}{n^{\frac{1}{2}+w}} \, dw.$$

	By inserting a smooth dyadic partition of unity into the $n$-sum we get 
	$$ |\mathcal{B}(1/2, f\otimes \chi_{8d})| \ll \Big \vert \sum_{N \text{-dyadic}} \int_{(3)} \gamma(w) \frac{(8d)^{w}-\mathcal{M}^{w}}{w^{2}} \sum_{n=1}^{\infty} \frac{\lambda_{f}(n) \chi_{8d}(n)}{n^{\frac{1}{2}+w}} G\left(\frac{n}{N}\right) \, dw \Big \vert,$$
	where $\gamma(w)=\frac{(2 \pi/\sqrt{q})^{-w} \, \Gamma\left({k}/{2}+w \right)}{\Gamma\left({k}/{2}\right)} $ and  $G$ is as given in Lemma \ref{G}   and 	by $N$ dyadic we mean that $N$ takes  the values $2^{H}, H=0,1,\cdots$.
	
	Consider  the $n$-sum
	\begin{align*}
		&\sum_{n=1}^{\infty} \frac{\lambda_{f}(n) \chi_{8d}(n)}{n^{\frac{1}{2}+w}} G\left(\frac{n}{N}\right) \\
		&= \sum_{n=1}^{\infty} \frac{\lambda_{f}(n) \chi_{8d}(n)}{n^{\frac{1}{2}+w}} V\left(\frac{n}{N}\right)G\left(\frac{n}{N}\right),
	\end{align*}
	with $V$ as in Remark \ref{V}. Using the  Mellin inversion we express $G$ as follows:
	$$G(x) = \frac{1}{2 \pi i} \int_{(\epsilon)} \widetilde{G}(s) \, x^{-s} \, ds,$$
	for any $\epsilon>0$, where $\widetilde{G}(s)$ is the Mellin transform of $G(x)$ given by
	$$\widetilde{G}(s)= \int_{0}^{\infty} G(x) \, x^{s-1} dx.$$
	Therefore the $n$-sum can be expressed as 
	$$\frac{1}{2 \pi i} \int_{(\epsilon)} \sum_{n=1}^{\infty} \frac{\lambda_{f}(n) \chi_{8d}(n)}{n^{\frac{1}{2}+s+w}} V\left(\frac{n}{N}\right) \, N^{s}\,\widetilde{G}(s) \, ds.$$
	By the change of  variable $s \mapsto s-w$ we arrive at 
	$$\frac{1}{2 \pi i} \int_{(\epsilon)} \sum_{n=1}^{\infty} \frac{\lambda_{f}(n) \chi_{8d}(n)}{n^{\frac{1}{2}+s}} V\left(\frac{n}{N}\right) \, N^{s-w}\,\widetilde{G}(s-w) \, ds.$$
	Next we split the $N$-sum as follows: 
	$$  \sum_{N \text{-dyadic}} = \sum_{\substack{N  \text{-dyadic} \\ N \leq \mathcal{M}}}+\sum_{\substack{N  \text{-dyadic} \\  \mathcal{M} <N\leq X}}+\sum_{\substack{N \, \text{-dyadic} \\ N >X}}.$$
	Thus we see that $|\mathcal{B}(1/2, f\otimes \chi_{8d})|^{2}$ is bounded by
	$$ \Big \vert \underbrace{ \sum_{\substack{N \, \text{-dyadic} \\ N \leq \mathcal{M}} } \mathcal{S}(N,\chi_{8d}) }_{\mathcal{B}_{1}(1/2, f\otimes \chi_{8d})}\Big \vert^{2} + \Big \vert \underbrace{\sum_{\substack{N\text{-dyadic} \\ \mathcal{M} < N \leq X}} \mathcal{S}(N,\chi_{8d})}_{\mathcal{B}_{2}(1/2, f\otimes \chi_{8d})} \Big \vert^{2}+\Big \vert \underbrace{\sum_{\substack{N \text{-dyadic} \\ N >X}} \mathcal{S}(N,\chi_{8d})}_{\mathcal{B}_{3}(1/2, f\otimes \chi_{8d})} \Big \vert^{2},$$
	where $\mathcal{S}(N,\chi_{8d})$ is given by
	$$ \frac{1}{2 \pi i}\int_{(3)}\gamma(w) \frac{(8d)^{w}-\mathcal{M}^{w}}{w^{2}} \\
	\int_{(\epsilon)} \sum_{n=1}^{\infty} \frac{\lambda_{f}(n) \chi_{8d}(n)}{n^{\frac{1}{2}+s}} V\left(\frac{n}{N}\right) \, N^{s-w}\,\widetilde{G}(s-w) \, ds \, dw.$$
	We next  estimate $$\mathcal{B}_i(X):= \sideset{}{^\star}{\sum}_{(d,2q)=1} \, |\mathcal{B}_{i}(1/2, f\otimes \chi_{8d})|^2 J\left(\frac{8d}{X}\right),$$
	for $i=1,2,3$ separately. Before embarking this job we record  the following estimates.
	$$ |\Gamma(k/2+w) \, (2 \pi/\sqrt{q})^{-w}| \ll (1+|w|)^{-20}, \, |\tilde{G}(s)| \ll (1+|s|)^{-20},$$
	for any $w$ and $s$ having bounded real parts. 
	\subsection{Estimates for  $ \mathcal{B}_{1}(X)$ }
	
		%
		We move line of integrations $\Re(w)=3$ to $\Re(w)=-1$ and $\Re(s)= \epsilon$ to $\Re(s)=0$, then by Cauchy's residue theorem we get that
		\begin{align*}
			&\mathcal{B}_{1}(1/2, f\otimes \chi_{8d}) = \log \left(\frac{8d}{\mathcal{M}}\right) \, \sum_{\substack{N- \text{dyadic} \\ N \leq \mathcal{M}}} \, \int_{(0)} \sum_{n=1}^{\infty} \frac{\lambda_{f}(n) \chi_{8d}(n)}{n^{\frac{1}{2}+s}} V\left(\frac{n}{N}\right) \, N^{s}\,\tilde{G}(s) \, ds \\
			&+  \Biggl\{\sum_{\substack{N- \text{dyadic} \\ N \leq \mathcal{M}}}\int_{(-1)} \frac{(2 \pi/\sqrt{q})^{-w} \, \Gamma\left(\frac{k}{2}+w \right)}{\Gamma\left(\frac{k}{2}\right)} \frac{(\frac{8d}{N})^{w}-(\frac{\mathcal{M}}{N})^{w}}{w^{2}} \\
			&\times \frac{1}{2 \pi i} \int_{(0)} \sum_{n=1}^{\infty} \frac{\lambda_{f}(n) \chi_{8d}(n)}{n^{\frac{1}{2}+s}} V\left(\frac{n}{N}\right) \, N^{s}\,\tilde{G}(s-w) \, ds \, dw \Biggr \}.
		\end{align*}
		Recall that $\mathcal{M}=X/(\log X)^{1000}$ and $X/2 \leq 8d \leq 2X$. Note that the number of  dyadic  $N$ such that $N \leq \mathcal{M}$ is at most $\log X$. We conclude that  $	|\mathcal{B}_{1}(1/2, f\otimes \chi_{8d})|$ is bound above by
		\begin{align*}
			& (\log X) (\log \log X) \sup_{N \leq \mathcal{M}} \int_{-\infty}^{\infty} \Big \vert \sum_{n=1}^{\infty} \frac{\lambda_{f}(n) \chi_{8d}(n)}{n^{\frac{1}{2}+it}} V\left(\frac{n}{N}\right) \Big \vert \, \frac{dt}{(1+|t|)^{20}} \\
			&+ \sum_{\substack{N- \text{dyadic} \\ N \leq \mathcal{M}}} \frac{N}{\mathcal{M}} \int_{-\infty}^{\infty} \int_{-\infty}^{\infty} \Big \vert \sum_{n=1}^{\infty} \frac{\lambda_{f}(n) \chi_{8d}(n)}{n^{\frac{1}{2}+it_{1}}} V\left(\frac{n}{N}\right) \Big \vert \frac{dt_{1}}{(1+|t_{2}|)^{20}} \, \frac{dt_{2}}{(1+|t_{1}-t_{2}|)^{20}}. 
		\end{align*}
		By  Cauchy-Schwartz inequality we get the following bound for $|\mathcal{B}_{1}(1/2, f\otimes \chi_{8d})|^{2} $
		\begin{align*}
			& (\log X)^{2} (\log \log X)^{2} \sup_{N \leq \mathcal{M}} \int_{-\infty}^{\infty} \Big \vert \sum_{n=1}^{\infty} \frac{\lambda_{f}(n) \chi_{8d}(n)}{n^{\frac{1}{2}+it}} V\left(\frac{n}{N}\right) \Big \vert^{2} \, \frac{dt}{(1+|t|)^{20}} \\
			&+ \sum_{\substack{N- \text{dyadic} \\ N \leq \mathcal{M}}} \frac{N}{\mathcal{M}} \int_{-\infty}^{\infty} \int_{-\infty}^{\infty} \Big \vert \sum_{n=1}^{\infty} \frac{\lambda_{f}(n) \chi_{8d}(n)}{n^{\frac{1}{2}+it_{1}}} V\left(\frac{n}{N}\right) \Big \vert^{2} \frac{dt_{1}}{(1+|t_{2}|)^{20}} \, \frac{dt_{2}}{(1+|t_{1}-t_{2}|)^{20}},
		\end{align*}
		where we have used $\sum_{\substack{N -\text{dyadic} \\ N \leq \mathcal{M}}}N/\mathcal{M} \ll 1$. Thus on applying  Proposition \ref{mainpropo} we see that   $\sideset{}{^\star}{\sum_d} \, |\mathcal{B}_{1}(1/2, f\otimes \chi_{8d})|^2 J\left(\frac{8d}{X}\right)$ is dominated by
		\[(\log X \log \log X)^{2} \sup_{N \leq \mathcal{M}} \left(X+N \log(2+\frac{N}{X})\right) + \sum_{\substack{N\, \text{dyadic} \\ N \leq \mathcal{M}}} \frac{N}{\mathcal{M}} \left(X+N \log(2+\frac{N}{X})\right)\]
		which is further bounded by
		$ X \, (\log X)^{2} (\log \log X)^{2} .$
		Hence we conclude that
		$$\sideset{}{^\star}{\sum}_{d} \, |\mathcal{B}_{1}(1/2, f\otimes \chi_{8d})|^2 J\left(\frac{8d}{X}\right) \ll X \, (\log X)^{2} (\log \log X)^{2}.$$ 
		\subsection{Estimates  for  $ \mathcal{B}_2(X)$}
		We move line of integrations $\Re(w)=3$ to $\Re(w)=1/\log X$ and $\Re(s)= \epsilon$ to $\Re(s)=0$. Then we see that $|\mathcal{B}_{2}(1/2, f\otimes \chi_{8d})| $ is at most 
		
		\begin{align*}
			\ll &\Big \vert \sum_{\substack{N- \text{dyadic} \\ \mathcal{M} < N \leq X}} \int_{(1/\log X)} \int_{(0)} \, \frac{(2 \pi/\sqrt{q})^{-w} \, \Gamma\left(\frac{k}{2}+w \right)}{\Gamma\left(\frac{k}{2}\right)} \frac{(\frac{8d}{N})^{w}-(\frac{\mathcal{M}}{N})^{w}}{w^{2}} \\
			& \times  \sum_{n=1}^{\infty} \frac{\lambda_{f}(n) \chi_{8d}(n)}{n^{\frac{1}{2}+s}} V\left(\frac{n}{N}\right) \, N^{s}\, \tilde{G}(s-w) \, ds \, dw \Big \vert. 
		\end{align*}
		We write the $w$-integral  as follows
		\begin{align*}
& \left(\frac{\sqrt{q}\mathcal{M}}{ 2 \pi N}\right)^{\frac{1}{\log X}}	\int_{-\infty}^{\infty} \frac{ \, \left(\frac{\sqrt{q}}{2\pi}\right)^{it_2}\Gamma\left(\frac{k}{2}+\frac{1}{\log X}+it_2 \right)}{\Gamma\left(\frac{k}{2}\right) (1/\log X+it_2)} h(t_2, X)	\tilde{G}(s-1/\log X+it_2)  dt_2,
		\end{align*}
Where
\[h(t_2, X):=\frac{({8d}/{\mathcal{M}})^{1/\log X+it_2}-1 }{(1/\log X+it_2)} \ll \log \log X. \]
We split  the above  integral as follows
\[ \int_{-\infty}^{\infty}= \int_{-1}^1+ \int_{-\infty  }^{-1}+ \int_1^{\infty  }.\]
Note that the contribution of the first part is bounded by $\log \log X$.  Thus we see that $	|\mathcal{B}_{2}(1/2, f\otimes \chi_{8d})|$ is bounded by
		\begin{align*}
		 (\log \log X)^2  \sum_{\substack{N \, \text{dyadic} \\ \mathcal{M} < N \leq X}}  \int_{1}^{\infty} \int_{\mathbb{R}}
			 \Big \vert \sum_{n=1}^{\infty} \frac{\lambda_{f}(n) \chi_{8d}(n)}{n^{\frac{1}{2}+it_{1}}} V\left(\frac{n}{N}\right) \Big \vert 
			  \frac{dt_{1}}{(1+|t_{2}|)^{20}}  \frac{dt_{2}}{(1+|t_{1}-t_{2}|)^{20}}.  
		\end{align*}
		We apply Cauchy-Schwartz inequality to get
		\begin{align*}
			|\mathcal{B}_{2}(1/2, f\otimes \chi_{8d})|^{2} \ll & (\log \log X)^{4} \, \sum_{\substack{N- \text{dyadic} \\ \mathcal{M} < N \leq X}} \, \int_{-\infty}^{\infty} \int_{-\infty}^{\infty} \\
			&\times \Big \vert \sum_{n=1}^{\infty} \frac{\lambda_{f}(n) \chi_{8d}(n)}{n^{\frac{1}{2}+it_{1}}} V\left(\frac{n}{N}\right) \Big \vert^{2} \, \frac{dt_{1}}{(1+|t_{2}|)^{20}} \, \frac{dt_{2}}{(1+|t_{1}-t_{2}|)^{20}}.
		\end{align*}
	 Now we  apply  Proposition \ref{mainpropo} and   $$\sum_{\substack{N -\text{dyadic} \\ \mathcal{M} < N \leq X}} 1 \ll \log (X/\mathcal{M}) \ll \log \log X$$  to conclude that 
		\begin{align*}
			&\sideset{}{^\star}{\sum}_{d} \, |\mathcal{B}_{2}(1/2, f\otimes \chi_{8d})|^2 J\left(\frac{8d}{X}\right) \\
			&  \ll  (\log \log X)^{4} \sum_{\substack{N- \text{dyadic} \\ \mathcal{M} < N \leq X}} \left( X+N \log\left(2+\frac{N}{X}\right)\right) \\
			&\ll X  (\log \log X)^{5}.
		\end{align*}
		\subsection{Estimates  for  $ \mathcal{B}_3(X)$}
		
		We move line of integrations $\Re(w)=3$ to $\Re(w)= 4$ and $\Re(s)= \epsilon$ to $\Re(s)=0$. Thus  we get 
		\begin{align*}
			|\mathcal{B}_{3}(1/2, f\otimes \chi_{8d})|^{2} \ll & \sum_{\substack{N- \text{dyadic} \\  N > X}} \left(\frac{X}{N}\right)^{4} \, \int_{-\infty}^{\infty} \int_{-\infty}^{\infty} \\
			&\times \Big \vert \sum_{n=1}^{\infty} \frac{\lambda_{f}(n) \chi_{8d}(n)}{n^{\frac{1}{2}+it_{1}}} V\left(\frac{n}{N}\right) \Big \vert^{2} \, \frac{dt_{1}}{(1+|t_{2}|)^{20}} \, \frac{dt_{2}}{(1+|t_{1}-t_{2}|)^{20}},  
		\end{align*}
		where we have used that $\sum_{\substack{N\text{-dyadic} \\  N > X}}\left(\frac{X}{N}\right)^{4} \ll 1 $. Thus we see that  
		\begin{align*}
			&\mathcal{B}_3(X)\ll \sum_{\substack{N- \text{dyadic} \\  N > X}}\left(\frac{X}{N}\right)^{4} \left(X+N \log\left(2+ \frac{N}{X}\right)\right) \ll X. 
		\end{align*}
		For   the second inequality we appeal to the inequality 
		$$\frac{X}{N} \log\left(2+\frac{N}{X}\right) \ll 1 $$ for $N>X$.  We conclude that
		$$\sideset{}{^\star}{\sum}_{d} \, |\mathcal{B}_{3}(1/2, f\otimes \chi_{8d})|^2 J\left(\frac{8d}{X}\right) \ll X.$$
		Thus we conclude the proof of  Proposition \ref{errorbound}.
		\section { Proof of Proposition \ref{aasymptotic} } \label{asmptoicsection}
		We are seeking an asymptotic formula for
		$$\sideset{}{^\star}{\sum}_{(d,2q)=1}  \left(\mathcal{A}(1/2, f \otimes \chi_{8d})\right)^{2} \, J\left(\frac{8d}{X}\right)=\sum_{(d,2q)=1}  \left(\mathcal{A}(1/2, f \otimes \chi_{8d})\right)^{2} \, J\left(\frac{8d}{X}\right) \sum_{a^{2}\mid d} \mu(a)$$
		Interchanging the order the summations, we  get the following expression
		\begin{align} \label{aasymptoitc}
		& 2 \left(\sum_{\substack{a \leq Y \\ (a,2q)=1}}+\sum_{\substack{a > Y \\ (a,2q)=1}}\right) \mu(a) \sum_{(d,2q)=1} J\left(\frac{8da^{2}}{X}\right) \\
		&	\times \left(1-i^{k} \eta \chi_{8d}(q)\right)\Bigg| \sum_{\substack{n=1 \\ (n,a)=1}}^{\infty} \frac{\lambda_{f}(n) \chi_{8d}(n)}{n^{1/2}} \, W\left( \frac{n}{\mathcal{M}}\right)
			\Bigg|^{2}, \notag 
		\end{align}
		where $Y=\log^{200}X$. 
		\subsection{Contribution of small $a$, $a\leq Y$}
		Let us denote the contribution of $a \leq Y$ to  \eqref{aasymptoitc} by $\mathcal{D(\leq Y)}$ which   is given by
		\begin{align*}
			2 \sum_{\substack{a \leq Y \\ (a,2q)=1}} \mu(a) \mathop{\sum \sum}_{\substack{n_{1},n_{2} \\ (n_{1}n_{2},2a)=1}} \frac{\lambda_{f}(n_{1}) \lambda_{f}(n_{2})}{\sqrt{n_{1}n_{2}}} \, W\left(\frac{n_{1}}{\mathcal{M}}\right) W\left(\frac{n_{2}}{\mathcal{M}}\right) \left(\mathcal{S}_{q^2}(\cdots )-i^{k} \eta \, \mathcal{S}_{q}(\cdots )\right),
		\end{align*}
		where
		\begin{equation} \label{d-sum}
			\mathcal{S}_{Q}(\cdots ):=	\mathcal{S}_{Q}(a;n_{1}n_{1}) = \sum_{(d,2)=1} \, \chi_{8d}(n_{1}n_{2}Q) \, J\left(\frac{8da^{2}}{X}\right)
		\end{equation}
		for $Q\in \{q,q^2\}$.  Using Lemma \ref{Poisson} we get 
		\begin{align*}
			\mathcal{S}_{Q}(a;n_{1}n_{1}) =\delta_{\square}(n_{1}n_{2}Q) \, \frac{X}{16a^{2}} \, \check{J} (0) \, \prod_{p \mid n_{1}n_{2}Q} \left(1-\frac{1}{p}\right) \\
			+ \frac{X}{16a^{2}} \sum_{ k \neq 0} (-1)^{k} \frac{G_{k}(n_{1}n_{2}Q)}{n_{1}n_{2}Q} \, \check{J}\left(\frac{X k}{16a^{2}n_{1}n_{2}Q}\right).
		\end{align*}
		Therefore we have $\mathcal{D(\leq Y)} = M + \mathcal{R}$, where  
		\begin{align*}
			M=&\frac{X \check{J}(0)}{8} \sum_{Q \in \{q,q^2\}} \, \varepsilon_{Q}  \sum_{\substack{a\leq Y \\ (a,2q)=1}} \frac{\mu(a)}{a^{2}} \mathop{\sum \sum}_{\substack{n_{1},n_{2} \\ n_{1}n_{2}Q=\square \\  (n_{1}n_{2},2a)=1}} \frac{\lambda_{f}(n_{1}\lambda_{f}(n_{2}))}{\sqrt{n_{1}n_{2}}}  \\
			& \times \prod_{p \mid n_{1}n_{2} Q}\left(1-\frac{1}{p}\right) \, W\left(\frac{n_{1}}{\mathcal{M}}\right) W\left(\frac{n_{2}}{\mathcal{M}}\right),
		\end{align*}
		and
		\begin{align*}
			\mathcal{R} =\frac{X }{8} \sum_{Q \in \{q,q^2\}} \, \varepsilon_{Q}  \sum_{\substack{a\leq Y \\ (a,2q)=1}} \frac{\mu(a)}{a^{2}}    \, T(a,Q)
		\end{align*}
		with 
		\begin{align*}
			T(a,Q)=&\sum_{\ell \neq 0} (-1)^{-\ell}  \mathop{\sum \sum}_{\substack{n_{1},n_{2} \\ (n_{1}n_{2},2a)=1}} \frac{\lambda_{f}(n_{1}) \lambda_{f}(n_{2})}{\sqrt{n_{1}n_{2}}}  \, \frac{G_{\ell}(n_{1}n_{2}Q)}{n_{1}n_{2}Q}  \\
			& \times  W\left(\frac{n_{1}}{\mathcal{M}}\right) W\left(\frac{n_{2}}{\mathcal{M}}\right)  \, \check{J}\left(\frac{X \ell}{16a^{2}n_{1}n_{2}Q}\right),
		\end{align*}
		where $\varepsilon_{q^{2}}=1$ and $\varepsilon_{q}=-i^{k} \eta $. Let
		$$T(a,Q)=-\left(T_{0}(a,Q)-\sum_{r \geq 1}T_{r}(a,Q) \right) $$
		where, for  $r\geq 0$, $	T_{r}(a,Q)$ is defined as 
		\begin{align*}
			& \sum_{\substack{\ell \neq 0 \\ \ell =2^{r} \ell^{\prime} \\ (2,\ell^{\prime})=1}} \mathop{\sum \sum}_{\substack{n_{1},n_{2} \\ (n_{1}n_{2},2a)=1}} \frac{\lambda_{f}(n_{1}) \lambda_{f}(n_{2})}{\sqrt{n_{1}n_{2}}}  \, \frac{G_{\ell}(n_{1}n_{2}Q)}{n_{1}n_{2}Q}  W\left(\frac{n_{1}}{\mathcal{M}}\right) W\left(\frac{n_{2}}{\mathcal{M}}\right)  \, \check{J}\left(\frac{X \ell}{16a^{2}n_{1}n_{2}Q}\right) \\
			=&\mathop{\sum \sum}_{\substack{n_{1},n_{2} \\ (n_{1}n_{2},2a)=1}}\underbrace{\sum_{\substack{(\ell,2)=1}}  \frac{\lambda_{f}(n_{1}) \lambda_{f}(n_{2})}{\sqrt{n_{1}n_{2}}}  \frac{G_{2^{\delta(r)} \ell}(n_{1}n_{2}Q)}{n_{1}n_{2}Q}   \check{J}\left(\frac{X 2^{r} \ell}{16a^{2}n_{1}n_{2}Q}\right) }_{\mathbb{T}(n_{1},n_{2})} W\left(\frac{n_{1}}{\mathcal{M}}\right) W\left(\frac{n_{2}}{\mathcal{M}}\right)
		\end{align*}
		where $\delta(r)=0$ if $2 \mid r $ and $\delta(r)=1$ if $2 \nmid r$. For the second equality we use the fact that $G_{4k}(n)=G_{k}(n)$ for odd $n$.
		
		\begin{lemma} \label{Trvalue}
			We have
			\begin{align*}
				&T_{r}(a,Q) \ll \, (\log X)^{2} \,\mathop{\sum \sum}_{N_{1},\,  N_{2}\text{-dyadic}} \left(1+\frac{N_{1}}{\mathcal{M}}\right)^{-6} \, \left(1+\frac{N_{2}}{\mathcal{M}}\right)^{-6} I(N_1,N_2),
			\end{align*}
			where $I(N_1,N_2)$ is given by 
			$$\mathop{\int \int \int \int } \Big \vert T(...) \Big \vert  \frac{dt_{1}}{\left(1+|t_{1}-t_{3}|\right)^{100}} \frac{dt_{2}}{\left(1+|t_{2}-t_{4}|\right)^{100}} \frac{dt_{3}}{\left(1+|t_{3}|\right)^{100}}\frac{dt_{4}}{\left(1+|t_{4}|\right)^{100}},$$
			and  $T(...)=T(N_{1},N_{2}; 2^{r} /16 a^{2},it_{1},it_{2})$ is  defined as in \eqref{Tvalue}.
		\end{lemma}
		
		\begin{proof}
			We use  the dyadic partition of unity to see that 
		\begin{align*}
		T_{r}(a,Q)&= \mathop{\sum \sum}_{N_{1}, N_{2} \text{-dyadic}} \mathop{\sum \sum}_{\substack{n_{1},n_{2} \\ (n_{1}n_{2},2a)=1}} \mathbb{T}(n_{1},n_{2})  \, W\left(\frac{n_{1}}{\mathcal{M}}\right) W\left(\frac{n_{2}}{\mathcal{M}}\right)  \\
		&  \times V\left(\frac{n_{1}}{N_{1}}\right) \, V\left(\frac{n_{2}}{N_{2}}\right) \,G\left(\frac{n_{1}}{N_{1}}\right) \, G\left(\frac{n_{2}}{N_{2}}\right),
	\end{align*}
			which can be further expressed as 
			\begin{align*}
				\mathop{\sum \sum}_{N_{1}, N_{2} - \, \text{dyadic}} &\frac{1}{(2 \pi i)^2}\int_{(\sigma_{1})}  \int_{(\sigma_{2})}  H(w_{1}) H(w_{2}) \frac{\mathcal{M}^{w_{1}}}{w_{1}^{2}} \frac{\mathcal{M}^{w_{2}}}{w_{2}^{2}}  \\
				&\times  \mathop{\sum \sum}_{\substack{n_{1},n_{2} \\ (n_{1}n_{2},2a)=1}} \frac{\mathbb{T}(n_{1},n_{2})  }{n_{1}^{w_{1}}n_{2}^{w_{2}}} \,V\left(\frac{n_{1}}{N_{1}}\right) \, V\left(\frac{n_{2}}{N_{2}}\right)\, G\left(\frac{n_{1}}{N_{1}}\right) \, G\left(\frac{n_{2}}{N_{2}}\right) dw_{1} dw_{2},
			\end{align*}
			where 
			$$H(w) = \frac{\Gamma\left(w+\frac{k}{2}\right) (2/\sqrt{q})^{-w}}{\Gamma(k/2)},$$
			and $\sigma_{1}, \sigma_{2} >0$. On applying the  Mellin inversion to  the above expression, we arrive at 
			\begin{align*}
				\mathop{\sum \sum}_{N_{1}, N_{2} - \, \text{dyadic}} &\frac{1}{(2 \pi i)^4}\int_{(\sigma_{1})}  \int_{(\sigma_{2})} \int_{(2\epsilon)} \int_{(2\epsilon)} H(w_{1}) H(w_{2}) \tilde{G}(u) \, \tilde{G}(v) \frac{\mathcal{M}^{w_{1}}}{w_{1}^{2}} \frac{\mathcal{M}^{w_{2}}}{w_{2}^{2}} \, N_{1}^{u}N_{2}^{v} \\
				& \times \mathop{\sum \sum}_{\substack{n_{1},n_{2} \\ (n_{1}n_{2},2a)=1}} \frac{\mathbb{T}(n_{1},n_{2})  }{n_{1}^{w_{1}+u}n_{2}^{w_{2}+v}}  \, \, du \, dv \, dw_{1} \, dw_{2} .
			\end{align*}
			We now make  the change of variables $u \to u-w_{1}$ and $v \to v-w_{2}$ in the  $u, v$-integrals to arrive at
			\begin{align*}
				\mathop{\sum \sum}_{N_{1}, N_{2} - \, \text{dyadic}} &\frac{1}{(2 \pi i)^4}\int_{(\sigma_{1})}  \int_{(\sigma_{2})} \int_{(2\epsilon)} \int_{(2\epsilon)} H(w_{1}) H(w_{2}) \tilde{G}(u-w_{1}) \, \tilde{G}(v-w_{2})  \\
				& \times 	\frac{(\mathcal{M} /N_{1})^{w_{1}}}{w_{1}^{2}} \frac{(\mathcal{M} /N_{2})^{w_{2}}}{w_{2}^{2}}  
				T(N_{1},N_{2}; 2^{r} /16 a^{2},u,v) \, N_{1}^{u}N_{2}^{v} \,    dw_{1} \, dw_{2}    \, du \, dv \,.
			\end{align*}
			where $	T(N_{1},N_{2}; \alpha,u,v)$ is given by 
			\begin{equation} \label{Tvalue}
				\sum_{\substack{(\ell,2)=1}}  \, \mathop{\sum \sum}_{\substack{n_{1},n_{2} \\ (n_{1}n_{2},2a)=1}}   \frac{\lambda_{f}(n_{1}) \lambda_{f}(n_{2})}{n_{1}^{1/2+u} n_{2}^{1/2+v}}  \, \frac{G_{2^{\delta(r)} \ell}(n_{1}n_{2}Q)}{n_{1}n_{2}Q}   \, \check{J}\left(\frac{\ell X \alpha}{n_{1}n_{2}Q}\right) V\left(\frac{n_{1}}{N_{1}}\right)  V\left(\frac{n_{2}}{N_{2}}\right).
			\end{equation}
			We now shift the contours in the $w_{1}$ and $w_{2}$ integrals. If $N_{1}>\mathcal{M}$ and $N_{2}>\mathcal{M}$, we move the contours to $\Re(\sigma_{1})=6$ and $\Re(\sigma_{2})=6$. If $N_{1} \leq \mathcal{M}$ and $N_{2} >\mathcal{M}$, we move the contours to $\Re(\sigma_{1})=-6$ and $\Re(\sigma_{2})=6$, and  in this case we get  a double pole at $w_{1}=0$  with the residue 
			$$H^{\prime}(0)\tilde{G}(u)-H(0) \tilde{G}^{\prime}(u)+H(0)\tilde{G}(u) \, \log \left(\frac{\mathcal{M}}{N_{1}}\right).$$
			The case $N_{1} >\mathcal{M}$ and $N_{2} \leq \mathcal{M}$ is similar.   In the remaining  case, i.e.,  $N_{1} \leq \mathcal{M}$ and $N_{2} \leq \mathcal{M}$ we move both the  contours to $\Re(\sigma_{1})=-6$ and $\Re(\sigma_{2})=-6$ and encounter  double poles at $w_{1}=0$ and $w_{0}=0$. Finally we move both the $u,v$ integrals to $\Re(u)=0$ and $\Re(v)=0$ and thus we get the lemma by estimating the $u,v$-integrals trivially. 
			%
		\end{proof}
		
		\subsection{Estimation of $T(N_{1},N_{2};\alpha,it_{1},it_{2})$}Set $\alpha^{\prime}= \alpha/2$ if $\delta(r)=1$ and $\alpha^{\prime}=\alpha$ if $\delta(r)=0$.  
		Recall that $T(N_{1},N_{2};\alpha^{\prime},it_{1},it_{2}) $ is given by
		$$\sum_{\substack{(\ell,2)=1}}  \, \mathop{\sum \sum}_{\substack{n_{1},n_{2} \\ (n_{1}n_{2},2a)=1}}   \frac{\lambda_{f}(n_{1}) \lambda_{f}(n_{2})}{n_{1}^{1/2+it_{1}} n_{2}^{1/2+it_{2}}}  \, \frac{G_{2^{\delta(r)} \ell}(n_{1}n_{2}Q)}{n_{1}n_{2}Q}   \, \check{J}\left(\frac{\ell X \alpha^{\prime}}{n_{1}n_{2}Q}\right) \, V\left(\frac{n_{1}}{N_{1}}\right) \, V\left(\frac{n_{2}}{N_{2}}\right).$$
		Using  the   expression  of $\check{J}$ and the Mellin inversion  theorem  in the above expression we arrive at 
		\begin{align} \label{Zintegration}
			&\frac{1}{(2 \pi i)^{3}} \int_{(\epsilon)} \int_{(2 \epsilon)} \int_{(2 \epsilon)} \tilde{V}(u) \tilde{V}(v) \, \tilde{J}(1-s) \Gamma(s) \\
			&\times \sideset{}{^\star}{\sum}_{(\ell_{1},2)} \left(\cos + \frac{\ell_{1}}{|\ell_{1}|} \sin \right)\left(\frac{\pi s}{2}\right) \left(\frac{Q}{|\ell_{1}| X \alpha^{\prime}}\right)^{s} \,  \notag\\
			& \times Z\left(\frac{1}{2}+it_{1}+u-s, \frac{1}{2}+it_{2}+v-s, s;\ell_{1},a\right) N_{1}^{u} \, N_{2}^{v} \, ds \, du \, dv, \notag 
		\end{align}
		with
		$$Z(\alpha,\beta, \gamma;\ell_{1},a)= \sum_{\substack{\ell_{2} \geq 1 \\ (\ell_{2},2)=1}} \mathop{\sum \sum}_{\substack{n_{1},n_{2} \\ (n_{1}n_{2},2a)=1}}   \frac{\lambda_{f}(n_{1}) \lambda_{f}(n_{2})}{n_{1}^{\alpha} n_{2}^{\beta} \ell_{2}^{2\gamma}}  \, \frac{G_{2^{\delta(r)} \ell_{1}\ell_{2}^{2}}(n_{1}n_{2}Q)}{n_{1}n_{2}Q},$$
		where we split the sum over $\ell$ as $\ell = \ell_{1} \ell_{2}^{2}$ with $\ell_{1}$ square free and $\ell_{2} \geq 1$.

		The following lemma is similar to  Lemma 2.5 of  \cite{Li}(with   minor changes).
		\begin{lemma} \label{dirichlet}
			Let $\ell_{1}$ be square free.  Let $m(\ell_{1})=\ell_{1}$ if $\ell_{1} \equiv \, 1 \, \rm mod \, 4$,  and $m(\ell_{1})=4\ell_{1}$ if $\ell_{1} \, \equiv, \, 2, 3 \rm mod\, \,4$. Then we have 
			$$Z(\alpha,\beta, \gamma;\ell_{1},a)= L\left(1/2+\alpha, f\otimes\chi_{m(\ell_{1})}\right) L\left(1/2+\beta, f\otimes\chi_{m(\ell_{1})}\right) \, Y\left(\alpha,\beta,\gamma;\ell_{1}\right),$$
			where $Y\left(\alpha,\beta,\gamma;\ell_{1}\right)$ is given by
			$$ \frac{Z_{2}\left(\alpha, \beta,\gamma;\ell_{1},a\right)}{\zeta(1+\alpha+\beta) L\left(1+2\alpha, \text{sym}^{2}f\right) L\left(1+2\beta, \text{sym}^{2}f\right) L\left(1+\alpha + \beta, \text{sym}^{2}f\right)} ,$$
			Here $Z_{2}(\alpha,\beta, \gamma;\ell_{1},a)$ is analytic in the region $\Re(\alpha),\Re(\beta) \geq -\delta/2$ and $\Re(\gamma) \geq 1/2+\delta$ for any $0<\delta <1/3$. Moreover, in this region we have $ Z_{2}(\alpha,\beta, \gamma;\ell_{1},a) \ll_{f,\delta} \, Q^{-1/2} \, \tau(a Q)$.  We  also  have the following Dirichlet series represention
			$$Y\left(\alpha,\beta,\gamma;\ell_{1}\right)= \mathop{\sum \sum \sum }_{r_{1},r_{2},r_{3}}\frac{C(r_{1},r_{2},r_{3})}{r_{1}^{\alpha}r_{2}^{\beta} r_{3}^{2\gamma}}, $$
			for some coefficients $C(r_{1},r_{2},r_{3})$.
		\end{lemma}
		
		\begin{proof}
			By  multiplicativity we have
			$$Z(\alpha,\beta, \gamma;\ell_{1},a) = \prod_{p} \, \mathcal{F}(p),$$
			where $\mathcal{F}(2)=1$ and for odd prime $p$, $\mathcal{F}(p)$ is given by
			\[
			\mathcal{F}(p)=
			\begin{cases}
				\sum_{n_{1},n_{2},\ell_{2} \geq 0}  \frac{\lambda_{f}(p^{n_{1}}) \lambda_{f}(p^{n_{2}})}{p^{n_{1}\alpha} p^{n_{2} \beta} p^{2 \ell_{2} \gamma} } \frac{G_{2^{\delta(r)}\ell_{1}p^{2\ell_{2}}} (p^{n_{1}+n_{2}})}{p^{n_{1}+n_{2}}} ,\quad  \, &\text{if} \, p\nmid aQ, \\
				\left(1-\frac{1}{p^{2\gamma}}\right)^{-1} \, &\text{if} \, \quad p \mid a, \, p \nmid Q, \\
				
				\sum_{n_{1},n_{2},\ell_{2} \geq 0}  \frac{\lambda_{f}(p^{n_{1}}) \lambda_{f}(p^{n_{2}})}{p^{n_{1}\alpha} p^{n_{2} \beta} p^{2 \ell_{2} \gamma} } \frac{G_{2^{\delta(r)}\ell_{1}p^{2\ell_{2}}} (p^{n_{1}+n_{2}+r_{p}})}{p^{n_{1}+n_{2}+r_{p}}} ,\quad  \, &\text{if} \, p\nmid a,  \,\, p^{r_{p}} || Q \\
				
				\sum_{\ell_{2} \geq 0} \frac{1}{p^{2\ell_{2}\gamma}} \frac{G_{2^{\delta(r)} \ell_{1} p^{2\ell_{2}}}(p^{r_{p}})}{p^{r_{p}}}, \quad \, &\text{if } \, p\mid (a, Q), \, \, p^{r_{p}} || Q.
				
			\end{cases}
			\]
			Let us assume that $\Re(\alpha), \Re(\beta) \geq -c$ and $\Re(\gamma) \geq 1/2+\delta$ for some $0< c<\delta <1/3$. For $p \nmid 2 \ell_{1}aQ$, we apply   Lemma \ref{Gausslemma} to $	\mathcal{F}(p)$ to arrive at 
			\begin{align*}
			&\sum_{\ell_{2} } \frac{1}{p^{2\ell_{2} \gamma}} \left(\sum_{h=0}^{\ell_{2}} \frac{\phi(p^{2h})}{p^{2h}} \sum_{\substack{i,j \\ i+j=2h}} \frac{\lambda_{f}(p^{i} )+\lambda_{f}(p^{j})}{p^{i \alpha+j\beta}} + \frac{\chi_{\ell_{1}}(p)}{p^{1/2}}  \sum_{\substack{i,j\\ i+j=2\ell_{2}+1}}\frac{\lambda_{f}(p^{i} )+\lambda_{f}(p^{j})}{p^{i \alpha+j\beta}}  \right) \\
				&= 1+\frac{\lambda_{f}(p) \chi_{\ell_{1}}(p)}{p^{1/2}} \left(\frac{1}{p^{\alpha}}+\frac{1}{p^{\beta}}\right) +O\left(\frac{1}{p^{1+2\delta-2c}}\right) \\
				&= \left(1-\frac{\lambda_{f}(p) \chi_{\ell_{1}}(p)}{p^{1/2+\alpha}}+\frac{\chi_{\ell_{1}}^{2}(p)}{p^{1+2\alpha}}\right)^{-1} \left(1-\frac{\lambda_{f}(p) \chi_{\ell_{1}}(p)}{p^{1/2+\alpha}}+\frac{\chi_{\ell_{1}}^{2}(p)}{p^{1+2\alpha}}\right)^{-1} \\
				&\times \left(1-\frac{A(p)}{p} \left(\frac{1} {p^{2\alpha}}+\frac{1}{p^{2\beta} }+\frac{1}{p^{\alpha+\beta}}\right) -\frac{1}{p^{1+\alpha+\beta}}+O\left(\frac{1}{p^{3/2-3c}}\right)\right) + O\left(\frac{1}{p^{1+2\delta-2c}}\right),
			\end{align*}
			and for $p\mid \ell_{1}$ and $p\nmid 2 aQ$, we have 
			$$\mathcal{F}(p)= 1-\frac{A(p)}{p} \left(\frac{1} {p^{2\alpha}}+\frac{1}{p^{2\beta} }+\frac{1}{p^{\alpha+\beta}}\right) -\frac{1}{p^{1+\alpha+\beta}}+O\left(\frac{1}{p^{3/2-3c}}\right) + O\left(\frac{1}{p^{1+2\delta-2c}}\right),$$
			where $A(n)$ are defined by 
			$$L(s, \text{sym}^{2}{f})=\sum_{n}\frac{A(n)}{n^{s}}, \quad \Re(s)>1.$$
			Choose  $c \leq \delta/2 <1/6$. 
			Note that 
			$\mathcal{F}(p) \ll p^{-1/2}$ for $p\mid Q$ and 
			$$\prod_{p \mid 2aQ} \left(1+\frac{C}{p^{1/2-c}}\right) \ll 2^{\omega(2aQ)} \ll \tau(aQ),$$
			where $\omega(n)$ is the number of prime factors of $n$ and $C$ is  any constant.
		\end{proof}
		The following lemma  is taken from  Li \cite[Lemma 5.7]{Li}. 
		\begin{lemma} \label{Lilemma}
			Suppose $\Re(s)\geq 3/5$ and write $u=-1/2+i\mu, v=-1/2+i\nu$ for real $\mu, \nu$.  We have 
			\begin{align*}
				\Big \vert &\mathop{\sum }_{r_{1},r_{2},r_{3}}\frac{C(r_{1},r_{2},r_{3})}{r_{1}^{1/2+u+it_{1}}r_{2}^{1/2+v+it_{2}} r_{3}^{2s}} G\left(\frac{r_{1}}{R_{1}}\right) G\left(\frac{r_{2}}{R_{2}}\right) \Big \vert  \\
				&\ll \left(1+|t_{1}|\right)^{1/10} \, \left(1+|t_{2}|\right)^{1/10} (1+|\mu|) (1+|\nu|) \, \exp\left(-c_{1} \sqrt{\log (R_{1}R_{2})}\right),
			\end{align*}
			where the  implied constant and $c_{1}$ may depend on $f$. 
		\end{lemma}
		
		\begin{remark}
			In fact,  Li \cite{Li} proved this lemma with $t_{1}=t, t_{2}=-t$, however his method also  proves the above lemma for any $t_{1}$ and $t_{2}$. 
		\end{remark}

		In view of Lemma \ref{dirichlet} the $u,v$-integrals in \eqref{Zintegration} are given by
		\begin{align*}
			& \sum_{r_{1},r_{2},r_{3}} \frac{C(r_{1},r_{2},r_{3})}{r_{1}^{1/2+it_{1}-s} r_{2}^{1/2+it_{2}-s}r_{3}^{2s}}  \mathop{\sum \sum}_{n_{1}, n_{2}} \frac{\lambda_{f}(n_{1} )\chi_{m(\ell_{1})}(n_{1})}{n_{1}^{1+it_{1}-s}} \frac{\lambda_{f}(n_{2}) \chi_{m(\ell_{1})}(n_{2})}{n_{2}^{1+it_{1}-s}} \, V\left(\frac{r_{1}n_{1}}{N_{1}}\right) V\left(\frac{r_{2}n_{2}}{N_{1}}\right) \\
			&= \mathop{\sum \sum }_{R_{1},R_{2}\, \text{dyadic}} \sum_{r_{1},r_{2},r_{3}} \frac{C(r_{1},r_{2},r_{3})}{r_{1}^{1/2+it_{1}-s} r_{2}^{1/2+it_{2}-s}r_{3}^{2s}}  G\left(\frac{r_{1}}{R_{1}}\right)\, G\left(\frac{r_{2}}{R_{2}}\right)\\
			&\times \mathop{\sum \sum}_{n_{1}, n_{2}} \frac{\lambda_{f}(n_{1} )\chi_{m(\ell_{1})}(n_{1})}{n_{1}^{1+it_{1}-s}} \frac{\lambda_{f}(n_{2}) \chi_{m(\ell_{1})}(n_{2})}{n_{2}^{1+it_{1}-s}} \, V_{1}\left(\frac{n_{1}R_{1}}{N_{1}}\right) \, V_{1}\left(\frac{n_{2}R_{2}}{N_{1}}\right) \, V\left(\frac{r_{1}n_{1}}{N_{1}}\right) V\left(\frac{r_{2}n_{2}}{N_{1}}\right) , 
		\end{align*}
		where 
		$$V_{1}(x)=G(4x)+G(2x)+G(x)+G(2x)+G(4x)$$
		which satisfies  $V_{1}(x)=1$ if $x \in [1/4,6]$.  We now apply Mellin inversion theorem  to separate variables $r_{i}$ and $n_{i}$ and make some  change of  variables to get the following expression for $T(N_{1},N_{2};\alpha,it_{1},it_{2})$ 
		\begin{align} \label{Tfinalexpres}
			&\frac{1}{(2 \pi i)^{3}} \int_{(\epsilon)} \int_{(2\epsilon)} \int_{(2 \epsilon)}\, \tilde{J}(1-s) \Gamma(s) \, \sideset{}{^\star}{\sum}_{(\ell_{1},2)} \left(\cos + \frac{\ell_{1}}{|\ell_{1}|} \sin \right)\frac{\pi s}{2} \left(\frac{QN_{1}N_{2}}{|\ell_{1}| X \alpha^{\prime}}\right)^{s} \tilde{V}(u+s)  \\
			&\times \tilde{V}(v+s) \mathop{\sum \sum }_{R_{1},R_{2}\, \text{dyadic}} \sum_{r_{1},r_{2},r_{3}} \frac{C(r_{1},r_{2},r_{3})}{r_{1}^{1/2+it_{1}+u} r_{2}^{1/2+it_{2}+v}r_{3}^{2s}}  G\left(\frac{r_{1}}{R_{1}}\right)\, G\left(\frac{r_{2}}{R_{2}}\right) \notag \\
			&\times \mathop{\sum \sum}_{n_{1}, n_{2}} \frac{\lambda_{f}(n_{1} )\chi_{m(\ell_{1})}(n_{1})}{n_{1}^{1+it_{1}+u}} \frac{\lambda_{f}(n_{2}) \chi_{m(\ell_{1})}(n_{2})}{n_{2}^{1+it_{1}+v}} \, V_{1}\left(\frac{n_{1}R_{1}}{N_{1}}\right) \, V_{1}\left(\frac{n_{2}R_{2}}{N_{1}}\right) \, N_{1}^{u}N_{2}^{v} \, du\, dv \,  ds. \notag 
		\end{align}
		
		In the following lemma we give an estimate for $T(N_{1},N_{2};\alpha,it_{1},it_{2})$.
		\begin{lemma} \label{Texpressionbound}
			We have
			$$\Big \vert T(N_{1},N_{2};\alpha,it_{1},it_{2}) \Big \vert      \ll \frac{Q\sqrt{N_{1}N_{2}}}{\alpha^{\prime} X} (1+|t_{1}|)^{\frac{17}{10}} \, (1+|t_{1}|)^{\frac{17}{10}} .$$
		\end{lemma}

		\begin{proof}
			We split the  $\ell_{1}$-sum as follows  $$T(...)=T_{(|\ell_{1}| \leq N_{1}N_{2}/X\alpha^{\prime})}(...)+T_{(|\ell_{1}|>N_{1}N_{2}/X\alpha^{\prime})}(...).$$ 
			First we consider the contribution of $|\ell_{1}| \leq \frac{QN_{1}N_{2}}{X \alpha^{\prime}}$.  In this case, in the expression we move $s,u,v$ contours in    \eqref{Tfinalexpres} to $\Re(s)=3/5$  and  $\Re(u), \Re(v)=-1/2$. Then we estimate  the resulting expression using $$|\Gamma(s) \cos s | + |\Gamma(s) \sin s |\ll |s|^{\Re(s)-1/2}$$ and Lemma \ref{Lilemma} to get the following  bound for $T_{(|\ell_{1}| \leq N_{1}N_{2}/X\alpha^{\prime})}(...)$ 
			\begin{align*}
				&  \left(\frac{\alpha^{\prime}X}{Q}\right)^{-3/5} (N_{1}N_{2})^{1/10}(1+|t_{1}|)^{1/10} (1+|t_{2}|)^{1/10} \sum_{R_{1},R_{2}\,  \text{dyadic}} \exp(-c_{1}\sqrt{\log(R_{1}R_{2})}) \\
				& \times \mathop{\int \int }_{-\infty}^{\infty} \frac{1}{(1+|\mu|)^{10}(1+|\nu|)^{10}} \sum_{|\ell_{1}| \leq \frac{QN_{1}N_{2}}{X\alpha^{\prime}}} \frac{1}{|\ell_{1}|^{3/5}} \Big \vert \sum_{n_{1}} \frac{\lambda_{f}(n_{1}) \chi_{m(\ell_{1})}(n_{1})}{n_{1}^{1/2+i\mu+it_{1}}} G\left(\frac{n_{1}}{N_{1}^{\prime}}\right) \\ 
				&\times \sum_{n_{2}} \frac{\lambda_{f}(n_{2}) \chi_{m(\ell_{1})}(n_{2})}{n_{2}^{1/2+i\nu+it_{1}}} G\left(\frac{n_{2}}{N_{2}^{\prime}}\right)  \Big \vert \, d\mu \, d \nu,
			\end{align*}
			where $N_{1}^{\prime} \asymp N_{1}/R_{1}$ and $N_{2}^{\prime} \asymp N_{2}/R_{2}$. We recall  Lemma 5.3 from \cite{Li} which asserts that
			\begin{equation} \label{modifiedlarge}
				\sideset{}{^\star}{\sum}_{M\leq |m| \leq 2M}\Big\vert \sum_{n} \frac{\lambda_{f}(n) \chi_{m}(n)}{n^{1/2+it}}  \, G\left( \frac{n}{N}\right)\Big \vert^{2} \ll_{f} (1+|t|)^{3} \, M \log (2+|t|).
			\end{equation}
			We now make a dyadic decomposition in $\ell_{1}$ sum and apply Cauchy-Scwartz inequality and then by the  above estimate we see thay the $\ell_{1}$-sum is bounded by
			$$ \left(\frac{QN_{1}N_{2}}{X\alpha^{\prime}}\right)^{2/5} (1+|t_{1}|)^{16/10} (1+|t_{2}|)^{16/10} (1+|\mu|)^{1/10} (1+|\nu|)^{1/10}.$$
			Therefore 
			\begin{align*}
				T_{(|\ell_{1}| \leq N_{1}N_{2}/X\alpha^{\prime})}(...) \ll \frac{Q\sqrt{N_{1}N_{2}}}{\alpha^{\prime} X} (1+|t_{1}|)^{\frac{17}{10}} \, (1+|t_{1}|)^{\frac{17}{10}} .
			\end{align*}
			In the second case when $|\ell_{1}| > \frac{QN_{1}N_{2}}{\alpha^{\prime} X}$ we move $s$ integral contour to $\Re(s)=6/5$ and $u, v$ integrals to $\Re(u), \Re(v)=-1/2$.  Therefore we see that $T_{(|\ell_{1}|>N_{1}N_{2}/X\alpha^{\prime})}(...)$ is dominated by
			\begin{align*}
				& \left(\frac{Q}{\alpha^{\prime} X}\right)^{6/5}(N_{1}N_{2})^{7/10}(1+|t_{1}|)^{1/10} (1+|t_{2}|)^{1/10} \sum_{R_{1},R_{2}\,  \text{dyadic}} \exp(-c_{1}\sqrt{\log(R_{1}R_{2})}) \\
				& \times \mathop{\int \int }_{-\infty}^{\infty} \frac{1}{(1+|\mu|)^{10}(1+|\nu|)^{10}} \sideset{}{^\star}{\sum}_{|\ell_{1}| > \frac{QN_{1}N_{2}}{X\alpha^{\prime}}} \frac{1}{|\ell_{1}|^{6/5}} \Big \vert \sum_{n_{1}} \frac{\lambda_{f}(n_{1}) \chi_{m(\ell_{1})}(n_{1})}{n_{1}^{1/2+i\mu+it_{1}}} G\left(\frac{n_{1}}{N_{1}^{\prime}}\right) \\ 
				&\times \sum_{n_{2}} \frac{\lambda_{f}(n_{2}) \chi_{m(\ell_{1})}(n_{2})}{n_{2}^{1/2+i\nu+it_{1}}} G\left(\frac{n_{2}}{N_{2}^{\prime}}\right)  \Big \vert \, d\mu \, d \nu.
			\end{align*}
			Consider  $\ell_{1}$ sum
			$$\sum_{\substack{J\, \text{dyadic} \\ J > \frac{QN_{1}N_{2}}{  \alpha^{\prime}X}} } \frac{1}{J^{6/5}} \sideset{}{^\star}{\sum}_{ |\ell_{1}| \sim J} \Big \vert \sum_{n_{1}} \frac{\lambda_{f}(n_{1}) \chi_{m(\ell_{1})}(n_{1})}{n_{1}^{1/2+i\mu+it_{1}}} G\left(\frac{n_{1}}{N_{1}^{\prime}}\right) \, \sum_{n_{2}} \frac{\lambda_{f}(n_{2}) \chi_{m(\ell_{1})}(n_{2})}{n_{2}^{1/2+i\nu+it_{1}}} G\left(\frac{n_{2}}{N_{2}^{\prime}}\right)  \Big \vert .$$
			By Cauchy-Schwartz inequality and \eqref{modifiedlarge} we see that the above expression  is dominated by
			
			\begin{align*}
				&(1+|t_{1}|)^{16/10} (1+|t_{2}|)^{16/10} (1+|\mu|)^{1/10} (1+|\nu|)^{1/10} \, \sum_{\substack{J\, \text{dyadic} \\ J > (QN_{1}N_{2})/\alpha^{\prime}X}} \frac{1}{J^{1/5}} \\
				&\ll \left(\frac{QN_{1}N_{2}}{\alpha^{\prime} X} \right)^{-1/5}(1+|t_{1}|)^{16/10} (1+|t_{2}|)^{16/10} (1+|\mu|)^{1/10} (1+|\nu|)^{1/10} . 
			\end{align*}
			Therefore we have 
			\begin{align*}
				T_{(|\ell_{1}| > N_{1}N_{2}/X\alpha^{\prime})}(...) \ll \frac{Q\sqrt{N_{1}N_{2}}}{\alpha^{\prime} X} (1+|t_{1}|)^{\frac{17}{10}} \, (1+|t_{1}|)^{\frac{17}{10}} .
			\end{align*}
			This concludes the proof  of the lemma.
		\end{proof}

		We have the following lemma which we gives estimates for $T_{r}(a,Q)$.
		
		\begin{lemma} \label{trbound}
			We have 
			$$T_{r}(a,Q) \ll \frac{Q a^{2}\mathcal{M}}{2^{r}X}(\log X)^{2} .$$
		\end{lemma}
		
		\begin{proof}
			From Lemma \ref{Texpressionbound} and Lemma \ref{Trvalue} we see that 
			\begin{align*}
				T_{r}(a,Q) &\ll \frac{Q a^{2}(\log X)^{2}}{2^{r}X} \mathop{\sum \sum}_{N_{1}, N_{2} - \, \text{dyadic}} \left(1+\frac{N_{1}}{\mathcal{M}}\right)^{-6} \, \left(1+\frac{N_{2}}{\mathcal{M}}\right)^{-6} \sqrt{N_{1}N_{2}} \\
				& \ll \frac{Q a^{2} \mathcal{M}(\log X)^{2}}{2^{r}X}. 
			\end{align*}
			This proves the lemma.
		\end{proof}

		\subsection{Estimation of  $\mathcal{R}$} Recall that 
		\begin{align*}
			\mathcal{R} =\frac{X }{8} \sum_{Q \in \{q,q^2\}} \, \varepsilon_{Q}  \sum_{\substack{a\leq Y \\ (a,2q)=1}} \frac{\mu(a)}{a^{2}}    \, T(a,Q)
		\end{align*}
		with
		$$T(a,Q)=-\left(T_{0}(a,Q)-\sum_{r \geq 1}T_{r}(a,Q) \right) .$$
		By Lemma \ref{trbound} we infer that
		$$T(a,Q)\ll  \frac{Q a^{2}\mathcal{M}}{X}(\log X)^{2} .$$
		Thus we have 
		\begin{equation} \label{errorestimate}
			\mathcal{R} \ll Y \mathcal{M}\ll \frac{X}{(\log X)^{10}}.
		\end{equation}
		\subsection{The main term $M$}Recall that
		\begin{align*}
			M=&\frac{X \check{J}(0)}{8} \sum_{Q \in \{q,q^2\}} \, \varepsilon_{Q}  \sum_{\substack{a\leq Y \\ (a,2q)=1}} \frac{\mu(a)}{a^{2}} \mathop{\sum \sum}_{\substack{n_{1},n_{2} \\ n_{1}n_{2}Q=\square \\  (n_{1}n_{2},2a)=1}} \frac{\lambda_{f}(n_{1} )\lambda_{f}(n_{2})}{\sqrt{n_{1}n_{2}}}  \\
			& \times  \prod_{p \mid n_{1}n_{2} Q}\left(1-\frac{1}{p}\right) \, W\left(\frac{n_{1}}{\mathcal{M}}\right) W\left(\frac{n_{2}}{\mathcal{M}}\right).
		\end{align*}
		Upon using 
		$$\sum_{\substack{a \leq Y\\ (a,2n_{1}n_{2})=1}} \frac{\mu(a)}{a^{2}} = \frac{8}{\pi^{2}} \prod_{p \mid n_{1}n_{2}}\left(1-\frac{1}{p^{2}}\right)^{-1} +O\left(\frac{1}{Y}\right),$$
		we infer that
		\begin{align*}
			M=\frac{X \check{J}(0)}{\pi^{2}} \sum_{Q \in \{q,q^2\}} & \varepsilon_{Q}  \mathop{\sum \sum}_{\substack{n_{1},n_{2} \\ n_{1}n_{2}Q=\square \\  (n_{1}n_{2},2)=1}} \frac{\lambda_{f}(n_{1}\lambda_{f}(n_{2}))}{\sqrt{n_{1}n_{2}}} \, \prod_{p \mid n_{1}n_{2} Q}\left(1-\frac{1}{p+1}\right) W\left(\frac{n_{1}}{\mathcal{M}}\right) \\  \times  W\left(\frac{n_{2}}{\mathcal{M}}\right) 
			&+O\left(\frac{X}{Y} \sum_{\substack{n_{1} n_{2}=\square}} \frac{\tau(n_{1}) \tau(n_{2})}{\sqrt{n_{1}n_{2}}} \vert W(n_{1}/\mathcal{M}) W(n_{2}/\mathcal{M})\vert\right).
		\end{align*}
		Note that the $O$-term is further dominated by  $\frac{X}{Y } (\log X)^{12}$.  Following calculations of  Petrow (see \cite[p. 1586-1588]{petrow}, we obtain that
		$$M= C_{f} \, \check{J}(0) \, X  \, (\log X)^{3} + O\left(X\, (\log X)^{2}\right),$$
		where $C_{f}$ is a constant depending only on $f$ and $C_{f}=0$ if and only if $i^{k}\eta =1$ and $q$ is square. 
		Therefore  we have
		\begin{equation} \label{smallycontri}
			D(\leq \mathcal{Y}) = C_{f} \, \check{J}(0) \, X  \, (\log X)^{3} + O\left(X\, (\log X)^{2}\right).
		\end{equation}
		
		\subsection{Contribution of large $a$, $a> Y$:}
		Note that
		\begin{align*}
			\sum_{\substack{n=1 \\ (n,a)=1}}^{\infty} \frac{\lambda_{f}(n) \chi_{8d}(n)}{n^{1/2}} \, W\left( \frac{n}{\mathcal{M}}\right) 
			&=\sum_{N \, \text{dyadic}}\frac{1}{(2\pi i)^{2}} \int_{(2)} \frac{\Gamma\left(\frac{k}{2}+w\right) }{(2\pi/\sqrt{q})^{w}\Gamma(\frac{k}{2})} \frac{\mathcal{M}^{w}}{w^{2}} \int_{(\epsilon)}  \\
			& \times \sum_{(n,a)=1} \frac{\lambda_{f}(n)}{n^{1/2+u}} \left(\frac{8d}{n}\right) V\left(\frac{n}{N}\right) N^{u-w} \, \tilde{G}(u-w) \, du \, dw. 
		\end{align*}
		Using Lemma \ref{coprimelargesieve} and calculations similar to Section \ref{propproofsec1} we infer that
		$$ \sum_{(d,2)=1} J\left(\frac{8da^{2}}{X}\right) \, \Big \vert \sum_{\substack{n=1 \\ (n,a)=1}}^{\infty} \frac{\lambda_{f}(n) \chi_{8d}(n)}{n^{1/2}} \, W\left( \frac{n}{\mathcal{M}}\right) \Big \vert^{2} \, \ll \frac{X (\log X)^{2}}{a^{2}} \, \tau(a)^{5}.$$
		Therefore 
		$$ \sum_{\substack{a>Y \\ (a,2)=1}  }  \sum_{(d,2)=1} J\left(\frac{8da^{2}}{X}\right) \, \Big \vert \sum_{\substack{n=1 \\ (n,a)=1}}^{\infty} \frac{\lambda_{f}(n) \chi_{8d}(n)}{n^{1/2}} \, W\left( \frac{n}{\mathcal{M}}\right) \Big \vert^{2} \, \ll \frac{X (\log X)^{2}}{Y^{1-\epsilon}} ,$$
		for any $\epsilon>0$. This is acceptable  as we have $Y=(\log X)^{200}$. Therefore by inserting this expression in \eqref{aasymptoitc} we have Proposition \ref{aasymptotic}. 
		\section*{\bf Acknowledgements}
		
		Authors are grateful to Xiannan Li for his encourgement and his interest in this work. Authors are also thankful to Ritabrata Munshi, Satadal Ganguly and Surya Ramana for their support.  The first and the  third author thank Alfr\'ed  R\'enyi Institute of Mathematics for providing an excellent research environment.   Authors would like to  thank the referee   for all the valuable  suggestions and comments, which  helped in improving the presentation of our article.

	\end{document}